\documentclass[11pt,a4paper,twoside]{article}
\usepackage{bm, amsmath, amssymb, amsthm}
\usepackage{graphicx}

\def\calf{{\cal F}}
\def\<{\langle}
\def\>{\rangle}
\def\eps{\varepsilon}

\def\RR{\mathbb{R}}
\def\Hb{H^\bot}
\def\Ht{H^\top}

\newcommand\const{\operatorname{const}}
\newcommand\tr{\operatorname{Tr}}
\newcommand\Div{\operatorname{div}}
\newcommand\id{\operatorname{id}}
\def\Ric{\operatorname{Ric}}
\def\vol{\operatorname{vol}}

\newtheorem{corollary}{Corollary}
\newtheorem{definition}{Definition}

\newtheorem{lem}{Lemma}

\newtheorem{thm}{Theorem}

\author{Vladimir Rovenski\footnote{Department of Mathematics, University of Haifa, Mount Carmel, 31905 Haifa,  Israel
        \newline e-mail: {\tt vrovenski@univ.haifa.ac.il} }
 }

\title{The weighted mixed curvature of a foliated manifold}

\begin{document}

\date{}

\maketitle


\begin{abstract}
 In this paper, we introduce  the weighted mixed (sectional, Ricci and scalar) curvature of a foliated
(and almost-product)
Riemannian manifold $(M,g)$ equipped with a vector field~$X$.
We define seve\-ral functions ($q$th Ricci type curvatures),
which ``interpolate" between the weighed sectional and Ricci curvatures.
The novel concepts of the ``mixed curvature-dimension" condition
and ``synthetic dimension of a distribution"
 allow us to update the estimate of the diameter of a compact Riemannian foliation
and to prove new splitting theorems for almost-product manifolds of nonnegative/nonpositive weighted mixed scalar curvature.
In the case of positive (and nonnegative) weighted mixed sectional curvature
we explore the weighted generalization of Toponogov's conjecture on totally geodesic foliations.


\end{abstract}

\section*{Introduction}

There are different ``weighted" curvatures for a Riemannian manifold $(M,g)$ endowed with a vector field~$X$,
e.g., when $X$ is a gradient of a density function $f:M\to\RR$.
The~notion of weighted scalar curvature comes up in Perelman's work and is related to his functionals for the Ricci flow.
The results for weighted Ricci curvature were first proven by Lihnerovicz, and later by Bakry--Emery and many others.
The \textit{weighted Ricci tensor} of the triple $(M,g,X)$,
\begin{equation}\label{E-BER}
 \Ric_X^N=\Ric +\,\frac12{\cal L}_X\,g -\frac1{N}\,X^\flat\otimes X^\flat
\end{equation}
(called also the $N$-Bakry--Emery--Ricci tensor)
has become important in
geometric analysis, for a brief overview see \cite{case,ww2}.
Here $\Ric$ is the Ricci tensor, ${\cal L}$ is the Lie derivative, and $N$ is the \textit{synthetic dimension},
that is also allowed to be infinite.
 The ``musical" isomorphisms $\flat$ and $\sharp$ ``lower" and ``raise" indices of tensors.
Definition \eqref{E-BER} arises for the $\Ric$
of a warped product
of $M$ of dimension $N>0$ with a manifold $H$, when the warping function $\phi=-\frac1N\log f$ and $X=\nabla f$.
The study of \eqref{E-BER} was motivated by the~\textit{curvature-dimension} condition CD$(c, N)$, which requires $\Ric_X^N\ge c$,
where $N$ is
an upper bound of the ``generalized dimension" of the weighted manifold and $c$ is a lower bound of the $\Ric$.
Some of
comparison results about the curvature bounds extend to the weighted case.
 In~\cite{kwy1}, the weighted sectional curvature is related to
 the torsion free connection
$\nabla^\alpha$ {projectively equiva\-lent} to the Levi-Civita connection, $\alpha=d\phi=X^\flat$ being a 1-form;
in this case, the Ricci tensor of $\nabla^\alpha$ is $\Ric_{X}^1$, see \eqref{E-BER} with~$N=1$.

\smallskip

Distributions on manifolds (i.e., subbundles of the tangent bundle) appear in various situ\-ations -- e.g. as fields of tangent planes of foliations or kernels of differential forms.
 {\it Totally geodesic} and {\it Riemannian foliations} have the simplest
extrinsic geometry of the leaves (respectively, the tangent or orthogonal distribution has zero second fundamental form),
and were investigated in a number of works.
 There is interest of geometers to problems of existence of adapted metrics on foliations and almost product manifolds
with given curvature properties.
 There are three kinds of sectional curvature for a foliation: \textit{tangential, transversal} and \textit{mixed}
 (denoted by $K_{\rm mix}$).
The mixed curvature is encoded in the {Riccati} and {Jacobi} equations along leaf geodesics.
 For constant $K_{\rm mix}$ the solutions of above equations
 (and the relative behavior of geodesics on nearby leaves) are well-known.
{Riemannian submersions} with totally geodesic fibers  have $K_{\rm mix}\ge 0$.
Splitting of foliated manifolds with $K_{\rm mix}=0$ is possible.

 In the paper, three types of weighted mixed curvature (sectional, $q$th Ricci and scalar) of foliated and almost product manifolds are defined,  the notions of ``synthetic dimension of a distribution" and the ``mixed curvature-dimension" condition are introduced, and natural  generalizations of several results (known for the case of $X=0$) are obtained.

\smallskip

Our main object is $(M^{n+\nu},g,X)$ equipped  with complementary orthogonal distributions ${\cal D}^\top$ and
${\cal D}^\bot$ of ranks $\dim{\cal D}^\top=\nu$ and $\dim{\cal D}^\bot=n$.
Let $^\top$ and $^\perp$ denote orthogonal projections onto ${\cal D}^\top$ and ${\cal D}^\bot$, respectively.
 We define several functions on $(M,g,X,{\cal D}^\top,{\cal D}^\bot)$, which ``interpolate" between the weighed sectional and Ricci curvatures; such functions on $(M,g)$ were introduced by H.\,Wu, and then studied by many geometers, see surveys in \cite{rov-m,rov2000}.
Let~$W^q$ be a subspace of ${\cal D}^\top_m$ spanned by $q\le\nu$ orthonormal vectors
$\{x_1,\dots, x_q\}$ at a point $m\in M$, and $y\in {\cal D}^\bot_m$ a unit vector.
Set $\Ric^{\top}_q(y;W):=\sum_{i=1}^q K(y,x_i)$ and $\Ric^{\top}:=\Ric^{\top}_\nu$.
The class of Riemannian  manifolds with $\Ric^{\top}_q>0$ is lager than class of manifolds with positive mixed sectional curvature.

\begin{definition}\rm
The ${\cal N}$-\textit{weighted mixed $q$th Ricci curvature} of $\{y;W\}$~is defined by
\begin{equation}\label{E-part-N-ric}
 \Ric^{\top,{\cal N}}_{q,X}(y;W) = \Ric^{\top}_q(y;W)
 +\frac q2\,{\cal L}_{X/\nu}\,g(y,y) +\frac{q\,\nu}{\cal N}\,g(X/\nu,y)^2,
\end{equation}
where ${\cal N}\in\RR$ is the \textit{synthetic dimension of} ${\cal D}^\top$.
For ${\cal N}=\nu$, the LHS of \eqref{E-part-N-ric} is $\Ric^{\top}_{q,X}(y;W)$,
and for $q=\nu$, the LHS of \eqref{E-part-N-ric} is
 $\Ric^{\top,{\cal N}}_{X}(y,y):=\Ric^{\top,{\cal N}}_{\nu,X}(y; {\cal D}^\top_m)$,
called the ${\cal N}$-\textit{weighted partial Ricci curvature}.
By the \textit{mixed curvature-dimension} condition, CD$^\top(c,{\cal N},q)$, we mean the inequality
\begin{equation}\label{E-MCD}
 \Ric^{\top,{\cal N}}_{q,X}\ge c.
\end{equation}
\end{definition}

Similarly, we define $\Ric^{\bot,N}_{q,X}(x;W)$ and $\Ric^{\bot}_{q,X}(x;W)$ for $W^q\subset{\cal D}^\bot$,
$\Ric^{\bot,{\cal N}}_{X}$,
$\Ric^{\bot}_{X}$
and condition CD$^\bot(c,{N},q)$.
 Notice~that
 $\Ric^{\top,{\cal N}}_{q,X}(y;W)<\Ric^{\top,{\cal N}'}_{q,X}(y;W)$
for
 ${\cal N}>{\cal N}'>0$,

\begin{equation}\label{Eq-Ric-N}
 \Ric^{\top,{\cal N}}_{q,X}(y;W) = \Ric^{\top}_{q,X}(y;W) +q\,\frac{{\cal N}-\nu}{\nu^2}\,g(X,y)^2.
\end{equation}

In~Section~\ref{sec:w-Ric}, we use \eqref{E-MCD} to estimate the diameter of Riemannian foliations.

\smallskip

The~\textit{weighted mixed sectional curvature} is the weighted sectional curvature of planes that non-trivially intersect
each of the distributions, see \eqref{E-part-N-ric} for $q=1$ and $W=\{x\}$,
\begin{equation}\label{E-w-sect}
 K^{\top,{\cal N}}_X(y,x)
 :=\Ric^{\top,{\cal N}}_{1,X}(y;\{x\})
 =K(y,x) +\big(\frac12\,{\cal L}_{X/\nu}\,g(y,y)+\frac\nu{\cal N}\,g(X/\nu,y)^2\,\big)\,x^2,
\end{equation}
in particular,
 $K^{\top}_X(y,x):=K^{\top,\nu}_X(y,x)$.
Similarly we define $K^{\bot,{N}}_X(x,y)$ and $K^{\bot}_X:=K^{\bot,n}_X$.
The $x$ and $y$ in \eqref{E-w-sect}
are placed in asymmetric way; generally, we have
$K^{\bot}_X(x,y)\ne K^{\top}_X(y,x)$.

\smallskip

One of the simplest curvature invariants of a pair $({\cal D}^\top,{\cal D}^\bot)$
is a function ${\rm S}_{\,\rm mix}:M\to\RR$,
\begin{equation}\label{E-Smix}
 {\rm S}_{\,\rm mix} =\tr_g \Ric^\top =\tr_g \Ric^\bot,
\end{equation}
see~\cite{r2010,rz1,wa1}, called the \textit{mixed scalar curvature}, i.e., an averaged mixed sectional curvature.
For~instance,
${\rm S}_{\,\rm mix}=\Ric(y,y)$
when $n=1$ and a unit vector field $y$ spans ${\cal D}^\bot$ locally.
In contrast to scalar curvature, ${\rm S}_{\,\rm mix}$ has strong relations with the extrinsic geometry, see Section~\ref{sec:scal},
and is involved in such research topics as the mixed Einstein--Hilbert action \cite{r2,rz-2} and prescribing the mixed scalar curvature on Riemann--Cartan (in particular, pseudo-Riemannian) manifolds \cite{rze-27b}.

Based on \eqref{E-part-N-ric} and \eqref{E-Smix}, we define the $(N,{\cal N})$-\textit{weighted mixed scalar curvature}
by
\begin{equation}\label{E-w-Smix-N}
 {\rm S}^{N,{\cal N}}_{\,{\rm mix},X}
 =\frac{1}{2}\tr_g\big(\Ric^{\bot,N}_{X} + \Ric^{\top,{\cal N}}_{X}\big)
 = {\rm S}_{\,\rm mix} + \frac{1}{2}\Div X +\frac1{2N}\,\|X^\top\|^2 +\frac1{2\cal N}\,\|X^\bot\|^2,
\end{equation}
where $N,{\cal N}\in\RR$ are \textit{synthetic dimensions of distributions} ${\cal D}^\bot$ and ${\cal D}^\top$, respectively.
For $N=n$ and ${\cal N}=\nu$, the LHS of \eqref{E-w-Smix-N} is ${\rm S}_{\,{\rm mix},X}$.
Notice~that
\begin{equation*}
 {\rm S}^{N,{\cal N}}_{\,{\rm mix},X} = {\rm S}_{\,{\rm mix},X}
 +\frac{{\cal N}-\nu}{2\nu\,{\cal N}}\,\|X^\bot\|^2 +\frac{N-n}{2n\,N}\,\|X^\top\|^2.
\end{equation*}
In~Section~\ref{sec:scal} we use the weighted mixed scalar curvature \eqref{E-w-Smix-N} to prove new integral formulas and update splitting theorems for almost-product manifolds.

\smallskip

Let $\calf^\nu$ be a totally geodesic foliation of $(M^{n+\nu},g)$,
and $R^\bot_{x}=(R_{\,\cdot,x}\,x)^\bot
\ (x\in T\calf)$
the Jacobi operator on ${\cal D}^\bot$.
 If~the leaves are closed and
 $R^\bot_x>0\ (x\ne0)$, then
$\nu<n$; otherwise,
any two of them will intersect.
D.\,Ferus~\cite{Fe} found a topological obstruction for existence of totally geodesic foliations
(and applied it to submanifolds with relative nullity):
\textit{if
$R^\bot_x\equiv k\id^\bot$ for some $k=\const>0$ and all unit $x\in T L$ on a complete leaf $L$, then $\nu<\rho(n)$}.
Here
$\rho(n)-1$ is the number of
linear independent vector fields on a sphere~$S^{n-1}$,
\[
 \rho(({\rm odd})\,2^{4b+c})=8b+2^c\ {\rm for\ some}\ b\ge 0,\ 0\le c\le3.
\]
Notice that $\rho(n)\le2\,\log_2n+2\le n$.
 Among Toponogov's many contributions to Riemannian geometry is the following
\textbf{conjecture} (see survey \cite{rov-m} and the bibliography~therein):
\textit{The
inequali\-ty $\nu<\rho(n)$ holds for totally geodesic foliations of closed Riemannian manifolds with
$K_{\rm mix}>0$}.

We introduce the \textit{weighted  ${\cal D}^\bot$-Jacobi operator} by
\begin{equation*}
 R^{\bot}_{X,x} = R^\bot_{x} + \big(\,\frac12\,{\cal L}_{X/n}\,g(x,x) + g(X/n,x)^2\,\big)\id^\bot,
 \quad x\in{\cal D}^\top,
\end{equation*}
and similarly define the \textit{weighted  ${\cal D}^\top$-Jacobi operator} $R^{\top}_{X,y}\ (y\in{\cal D}^\bot)$.
Set
\[
 \Ric^\bot_X(x,x)=\tr_g R^{\bot}_{X,x},\quad
 \Ric^\top_X(y,y)=\tr_g R^{\top}_{X,y}.
\]

\smallskip

In Section~\ref{sec:w-Top} we explore the following \textbf{``weighted" Toponogov's conjecture}:

\smallskip

\textit{The inequality $\nu<\rho(n)$ holds for a totally geodesic foliation $\calf^\nu$ of a closed manifold $(M^{n+\nu},g,X)$ under assumption
$ R^{\bot}_{X,x} > \|X^\top/n\|^2\,{\rm id}^\bot $ for all unit vectors $x\in T\calf$}.

\smallskip\noindent\
Since $R^{\bot}_{X,x}>0\ (x\ne0)$
yields $R^\bot_x > -\frac12\,{\cal L}_{X/n}\,g(x,x)\id^\bot$, then $R^\bot_x$ can be negative somewhere.

\section{The weighted mixed $q$th Ricci curvature}
\label{sec:w-Ric}

The weighted curvature appears in the formula for the second variation of energy of a path.

\begin{lem} Let $\calf $ be a Riemannian foliation of $(M,g,X)$,
and $\bar\gamma:[a,b]\times(-\eps,\eps)\to M$ be a variation of the geodesic $\gamma(t)=\bar\gamma(t,0)$
and the variation vector field on $\gamma$, $x(t)=\partial_s\bar\gamma |_{s=0}$ belongs to $T\calf$.
Then
we have the following index form on a geodesic $\gamma$:
\begin{equation}\label{E-2var-w}
 {\cal I}(x,x) = \int_a^b \big( |\dot x - g(\dot\gamma,X)\,x|^2 -K^\top_X({\dot\gamma},x)\,|x|^2
 \big)\,dt
 +g(\dot\gamma,X) |x|^2\,|_a^b .
\end{equation}
\end{lem}

\begin{proof}
Recall that
for a Riemannian foliation any geodesic started orthogonally to a leaf remains to be orthogonal to the leaves.
Thus, the proof
is similar to the proof of \cite[Proposition~5.1]{ww1}.
\end{proof}

\noindent
Let ${\rm diam}\,\calf$ be the maximal distance between the~leaves of a foliation $\calf$ of $(M,g,X)$.

\begin{lem}[see \cite{rov-m}]\label{P-01}
Let $V_1,\,V_2$ are subspaces in $\Bbb R^l$, $\dim V_1=\dim V_2$.
Then there exist orthonormal bases $\{a_i\}\subset V_1,\ \{b_i\}\subset V_2$
(which correspond to extremal values of angle between given subspaces) with the
property $a_i\,\bot\,b_j\ (i\ne j)$.
\end{lem}

\begin{thm}\label{T-3.29}
Let $(M^{n+\nu},g,X)$ be endowed with a Riemannian foliation $\calf^\nu$ with compact leaves such that
{\rm CD}$^\top(c,{\cal N},q)$, see \eqref{E-MCD}, holds
for some ${\cal N}\le \nu$, $c>0$ and $1\le q\le\nu$.
Then
\begin{equation}\label{Eq-w-Ric}
 ({\rm diam}\,\calf)^2 \le \frac{2q\|X^\bot\|}{c{+}q\|X^\bot\|^2}
 +\left\{\begin{array}{cc}
  \frac{2q}c\,\|h_\calf\| +\frac{\pi^2}4 & {\rm if} \ \nu\le n-1,\\
  \!\frac{2q}c\,\|h_\calf\| +(q-\nu+n-1)\frac{\pi^2}{4 c} & \ {\rm if} \ n-1<\nu< n+q-1,\\
  \frac{2q}c\,\|h_\calf\| & {\rm if} \ \nu\ge n+q-1.\\
 \end{array}\right.
\end{equation}
\end{thm}

\begin{proof}
Consider two leaves $L_1,L_2$ with distance $l={\rm dist}\,(L_1,L_2)$,
which is reached at points $m_1\in L_1$ and $m_2\in L_2$.
By the first variational formula of arc-length, the shortest geodesic
$\gamma(t)\ (0\le t\le 1)$ with length $l$ between $m_1,m_2$ is orthogonal to
$L_1$ and $L_2$. Since $\calf$ is a Riemannian foliation, $\gamma$ intersect the leaves orthogonally for all $t\in(0,1)$.

 Assume the second case: $n-1<\nu<n-1+q$, see \eqref{Eq-w-Ric}. Then the parallel displacement of
$T_{m_1}L_1$ along $\gamma$ will intersect $T_{m_2}L_2$ by $q'-$dimensional
subspace $V_2$, where $\nu-n+1\le q'<q$. The inverse image of $V_2$ in
$T_{m_1}L_1$ we denote by $V_1$.
For small $l$, let $T_{m_1}L_1=V_1\oplus V_1'$
be the orthogonal decomposition where the parallel image of $V_1'$ is uniquely
projected onto $T_{m_2}L_2$ (denote its orthogonal projection in $T_{m_2}L_2$ by $V_2'$).
Let vectors $e_1,\dots,e_{q'}$ form an orthonormal basis of $V_1$ and continue
them to parallel vector fields $\bar e_1,\dots,\bar e_{q'}$ along $\gamma$.
Obviously, $\bar e_1(m_2),\dots,\bar e_{q'}(m_2)$ belong to $V_2$.
Let vectors $a_1,\dots,a_s$ (where $s=q-q'=\dim V_1'$) form an orthonormal basis of
$V_1'$ and vectors $b_1,\dots,b_s$ form an orthonormal basis of $V_2'$, and
continue them to parallel vector fields $\bar a_1,\dots,\bar a_s$ and
$\bar b_1,\dots,\bar b_s$ along $\gamma$.

Consider the field of parallel planes $\sigma_i(t)$ along $\gamma$, spanned by
vectors $\bar a_i(t),\,\bar b_i(t)$. Assume, that $\{a_i\},\,\{b_i\}$
correspond to extremal angles between $V_1'$ and parallel image of $V_2'$, see Lemma~\ref{P-01}.
Then
$\sigma_i(t)\,\bot\,\sigma_j(t)$ for $i\ne j$.
We take the unit vector $\tilde b_i(t)\in\sigma_i(t)$ such that $g(\bar a_i,\tilde b_i(t))=0$.
One may choose $b_i$ and $\tilde b_i(t)$ with the properties
$g(\bar a_i,\bar b_i)\ge 0$ and $g(\bar b_i,\tilde b_i(t))\ge 0$.
Let us introduce the unit vector fields
$x_i(t)=(\cos\theta_it)\,\bar a_i+(\sin\theta_it)\,\tilde b_i(t)$
along $\gamma$,
where $\theta_i=\arccos(\bar a_i,\bar b_i)\in[0,\frac\pi 2]$.
Note that $g(x_i(t),x_j(t))=0\ (i\ne j)$, and $g(\dot x_i(t),x_i(t))=0$.

We have $q'+s=q$.
 Using the 2nd variation of ${\cal E}$ of $\gamma$, \eqref{E-2var-w}, along $x_i(t)$ and $\bar e_j$, we obtain
\begin{eqnarray}\label{E-3.65}
\nonumber
 && \hskip-11mm
 {\cal E}_{x_i}''(0) = (h_L(b_i,b_i),{\dot\gamma(1)}/l){-}(h_L(a_i,a_i),{\dot\gamma(0)}/l) +\theta_i^2 \\
\nonumber
 && -\,l^2\int_0^1 [K^\top_X({\dot\gamma},x_i(t)) +g(\dot\gamma/l,X)^2]\,dt +2g(\dot\gamma/l,X)\,|_0^1\ge 0,\\
\nonumber
 && \hskip-11mm
 {\cal E}_{\bar e_j}''(0) = (h_L(\bar e_j,\bar e_j),{\dot\gamma(1)}/l) -(h_L(e_j,e_j),{\dot\gamma(0)}/l) \\
 && -\,l^2\int_0^1 [K^\top_X({\dot\gamma},\bar e_j)+g(\dot\gamma/l,X)^2]\,dt +2g(\dot\gamma/l,X)\,|_0^1\ge 0.
\end{eqnarray}
Since $s=q-q'\le q-\nu+n-1$, $\sum_i\theta_i^2\le \frac{\pi^2}4s$, we have
\begin{eqnarray*}
 \sum\nolimits_{i=1}^{q'}|(h_L(b_i,b_i),{\dot\gamma(1)}/l)-(h_L(a_i,a_i),{\dot\gamma(0)}/l)|\le2q'||h_L||,\\
 \sum\nolimits_{j=1}^s|(h_L(\bar e_j,\bar e_j),{\dot\gamma(1)}/l)-(h_L(e_j,e_j),{\dot\gamma(0)}/l)|\le2s||h_L||.
\end{eqnarray*}
By \eqref{Eq-Ric-N} and condition for $\Ric^{\top,{\cal N}}_{q,X}$ we get
 $\Ric^{\top}_{q,X}(\dot\gamma;W) \ge c$
for $W$ spanned by $x_i(t)$ and $\bar e_j$;
hence,
\[
 \sum\nolimits_{i=1}^{q'}K^\top_X({\dot\gamma},x_i(t))+\sum\nolimits_{j=1}^{q-q'}K^\top_X({\dot\gamma},\bar e_j)\ge c.
\]
Then from \eqref{E-3.65} follows
\[
 l^2(c+q\|X^\bot\|^2)\le 2q||h_L||+(q-\nu+n-1)\,{\pi^2}/4+2q\,\|X^\bot\|,
\]
which completes the proof of second formula in \eqref{Eq-w-Ric}.
The other two cases are similar.
\end{proof}

\begin{corollary}
Let $(M^{n+\nu},g,X)$ be
endowed with a compact totally geodesic foliation $\calf^\nu$
with condition $\Ric^{\top,{\cal N}}_{q,X}>0$ for some ${\cal N}\le \nu$ and $1\le q\le\nu$ along some leaf, and let $X$ be tangent to the leaves.
Then $\nu < n+q-1$.
\end{corollary}

Similar results are true for foliated pseudo-Riemannian manifolds.

\section{The weighted mixed scalar curvature}\label{sec:scal}

Define tensors for one of distributions,
 ${\cal D}^\top$; similar tensors for ${\cal D}^\bot$
are defined using $\,^\bot$ notation.
 Let
 $T^\top, h^\top:{\cal D}^\top\times{\cal D}^\top\to{\cal D}^\bot$
be the integrability tensor and the 2nd fundamental form of~${\cal D}^\top$,
\begin{equation*}
  T^\top(u,v):=(1/2)\,[u,\,v]^\perp,\quad h^\top(u,v): = (1/2)\,(\nabla_u v+\nabla_v u)^\perp.
\end{equation*}
Then $\Ht =\tr_g h^\top$ is the mean curvature vector of ${\cal D}^\top$.
A distribution ${\cal D}^\top$ is called \textit{totally umbilical}, \textit{harmonic}, or \textit{totally geodesic}, if
 $h^\top=\frac1\nu\,\Ht\!\cdot g^\top$, $\Ht =0$ or $h^\top=0$, respectively.

The~Weingarten operator $A^\top$ (of ${\cal D}^\top$)
 and the operator ${T}^{\top\sharp}$ are defined by
\[
 g(A^\top_Z u,v)= g(h^\top(u^\top,v^\top),w^\bot),\quad
 g({T}^{\top\sharp}_w u,v)=g(T^\top(u^\top,v^\top),w^\bot).
\]
The~local adapted orthonormal frame $\{E_a,\,{\cal E}_{i}\}$, where $\{E_a\}\subset{{\cal D}^\top}$, always exists on~$M$.
We~use inner products of tensors, e.g.
\begin{equation*}
 \|h^\top\|^2=\sum\nolimits_{\,i,j} \,g(h^\top({\cal E}_i,{\cal E}_j),h^\top({\cal E}_i,{\cal E}_j)),
 \quad
 \|T^\top\|^2=\sum\nolimits_{\,i,j} \,g(T^\top({\cal E}_i,{\cal E}_j),T^\top({\cal E}_i,{\cal E}_j)).
\end{equation*}

\subsection{Integral formulas}

Integral formulae for foliated manifolds relate extrinsic geometry of the leaves and curvature and provide obstructions for existence of foliations with given geometric properties.
Following ideas of \cite{wa2}, we consider singular distributions, that is those defined outside
a ``singularity set"~$\Sigma$, a~finite union of pairwise disjoint closed submanifolds of codimension $\ge k$
under
assumption that improper integrals $\int_M \|\xi\|^s\,{\rm d}\vol_g$ converge for suitable
vector fields $\xi$ defined on $M\smallsetminus\Sigma$.

\begin{lem}[see \cite{wa2}]\label{L-sing1}
If $(k-1)(s-1)\ge1$ and $\xi$ is a vector field on $M\setminus\Sigma$ such that
$\|\xi\|\in L^s(M,g)$
then \eqref{E-Div-Th} holds.
\end{lem}

The Divergence Theorem for a vector field $\xi$ on  $(M,g)$ states that
\begin{equation}\label{E-Div-Th}
 \int_M (\Div \xi)\,{\rm d}\vol_g=0,
\end{equation}
 when either $\xi$ has compact support or $M$ is closed.
 The divergence of the vector field $\Ht+\Hb$ on a Riemannian almost product manifold
was given explicitly in \cite{wa1}:
\begin{equation}\label{E-PW-IF}
 \Div(\Ht+\Hb) = {\rm S}_{\,\rm mix}-\|T^\top\|^2-\|T^\bot\|^2+\|h^\top\|^2 +\|h^\bot\|^2-\|\Ht\|-\|\Hb\|^2.
\end{equation}
The~${\cal D}^\bot$-\textit{divergence} of a vector field $\xi$ is defined by
$\Div^\perp\xi=\sum\nolimits_{i}
g(\nabla_{i}\,\xi, {\cal E}_i)$,
and
we have
\begin{equation}\label{E-Div-DD}
 {\Div}^\top (\xi^\top) = \Div \xi^\top + g(\xi,\,\Hb),\quad
 {\Div}^\bot (\xi^\bot) = \Div \xi^\bot + g(\xi,\,\Ht).
\end{equation}
By \eqref{E-PW-IF} and using
 $\Div(\Hb)=\Div^\top(\Hb)-\|\Hb\|^2$
 and
 $\Div (\Ht)=\Div^\bot (\Ht)-\|\Ht\|^2$,
 we~get
\begin{equation}\label{E-div-IF2}
 \Div^\top(\Hb) +\Div^\bot (\Ht) =
 {\rm S}_{\,\rm mix} +\|h^\bot\|^2 +\|h^\top\|^2 -\|T^\bot\|^2 -\|T^\top\|^2.
\end{equation}

In the next theorems we are based on \eqref{E-PW-IF} and \eqref{E-div-IF2} and extend results in \cite{wa1}.

\begin{thm}
Let $(M,g)$ be a closed
Riemannian manifold endowed with complementary ortho\-gonal distributions ${\cal D}^\top$
and ${\cal D}^\bot$ defined on the complement to the ``set of singularities" $\Sigma$ with ${\rm codim}\,\Sigma\ge k$, and a vector field $X$ such that
$\|\xi\|_g\in L^s(M,g)$, where $\xi=\Hb +\Ht +\frac12\,X$ and  $(k-1)(s-1)\ge1$.
Then for all $N,{\cal N}\ne0$ the following integral formula holds:
\begin{eqnarray*}
 && \int_{M}\!\big\{{\rm S}^{N,{\cal N}}_{{\rm mix},X} -\|T^\top\|^2 -\|T^\bot\|^2 +\|h^\top\|^2 +\|h^\bot\|^2
 -\|\Ht\|^2 -\|\Hb\|^2 \\
 && \qquad -\,\frac1{2N}\,\|X^\top\|^2 -\frac1{2{\cal N}}\,\|X^\bot\|^2 \big\}\,{\rm d}\vol_g = 0 .
\end{eqnarray*}
\end{thm}

\begin{proof} This follows from \eqref{E-PW-IF}, \eqref{E-w-Smix-N} and Lemma~\ref{L-sing1}.
\end{proof}

\begin{definition}\rm
We say that $(M',g')$ is a \textit{leaf of a distribution} ${\cal D}$ on $(M,g)$ if $M'$ is a submanifold of $M$ with induced metric $g'$ and
$T_mM'={\cal D}_m$ for any $m\in M'$.
\end{definition}

\begin{thm}
Let $(M,g)$ be a closed Riemannian manifold endowed with complementary orthogo\-nal distributions ${\cal D}^\top$ and ${\cal D}^\bot$
with $\Ht=0$,
defined on the complement to the ``set of singularities" $\Sigma$ with ${\rm codim}\,\Sigma\ge k$, and a vector field $X\in\mathfrak{X}^\top$ such that $\|\xi_{\,|M'}\|_g\in L^s(M',g')$ for all leaves $(M',g')$ of ${\cal D}^\top$,
where $\xi=\Hb +\frac12\,X$ and $(k-1)(s-1)\ge1$.  Then for all $N,{\cal N}\ne0$ the following
integral formula
holds:
\begin{equation*}
 \int_{M'}\!\big\{{\rm S}^{N,{\cal N}}_{{\rm mix},X}
 -\|T^\bot\|^2 +\|h^\top\|^2 +\|h^\bot\|^2 +\frac12\,g(X,H^\bot) -\frac1{2N}\,\|X\|^2
  \big\}\,{\rm d}\vol_{g'} = 0 .
\end{equation*}
\end{thm}

\begin{proof} This follows from \eqref{E-div-IF2}, \eqref{E-w-Smix-N} and Lemma~\ref{L-sing1}.
\end{proof}

\subsection{Splitting of almost product manifolds}
\label{sec:scal2}

Applying S.T.Yau version of Stokes' theorem on a complete open
$(M,g)$ yields the following.

\begin{lem}[see Proposition~1 in \cite{csc2010}]\label{L-Div-1}
 Let $(M,g)$ be a complete open Riemannian manifold endowed with a vector field $\xi$
 such that $\Div\xi\ge0$. If the norm $\|\xi\|_g\in L^1(M,g)$ then $\Div \xi\equiv0$.
\end{lem}

\noindent
In~Section~\ref{sec:scal2} we
update some splitting theorems
to the case of
almost-product manifolds of
nonnegative/nonpositive weighted mixed scalar curvature.

\subsubsection{Harmonic distributions}

\begin{thm}
Let $(M,g)$ be a complete open {\rm (}or closed{\rm )} Riemannian manifold
endowed with complementary
orthogo\-nal integrable distributions $({\cal D}^\top,{\cal D}^\bot)$
and a vector field $X\in\mathfrak{X}^\top$ obeying conditions
$g(X,H^\bot)=0$ and $\|X_{\,|M'}\|_{g'}\in L^1(M',g')$
for all leaves $(M',g')$ of ${\cal D}^\top$.
Suppose that ${\cal D}^\top$ is harmonic and $\,{\rm S}^{N,{\cal N}}_{\,{\rm mix},X}\ge0$ for some $N<0$ and ${\cal N}\ne0$.
Then $M$~splits and $X=0$.
\end{thm}

\begin{proof}
By conditions and \eqref{E-w-Smix-N}, \eqref{E-Div-DD} and \eqref{E-div-IF2}, we have
\begin{equation}\label{E-th4}
 \Div^\top (\,H^\bot+\frac12\,X) = {\rm S}^{N,{\cal N}}_{\,{\rm mix},X} +\|h^\bot\|^2 +\|h^\top\|^2
 -\frac1{2N}\,\|X\|^2.
\end{equation}
Applying Lemma~\ref{L-Div-1} to each leaf,
we get $\Div^\top (\,H^\bot+\frac12\,X)$
when ${\rm S}^{N,{\cal N}}_{\,{\rm mix},X}\ge0$  and $N<0$;
thus, $h^\top=0=h^\bot$ and $X=0$. By~de Rham decomposition theorem, $(M,g)$ splits.
\end{proof}

\begin{corollary}\label{C-w2}
Let $(M,g,X)$ be
endowed with two complementary orthogo\-nal integrable distributions $({\cal D}^\top,{\cal D}^\bot)$
 and a vector field $X\in\mathfrak{X}^\top$.
Suppose that
${\cal D}^\bot$ is harmonic.
Then~${\cal D}^\top$ has no compact harmonic leaves $M'$ with ${\rm S}^{N,{\cal N}}_{\,{\rm mix},X\,|M'}>0$ for some $N<0$ and ${\cal N}\ne0$.
\end{corollary}

\begin{thm}
Let $(M,g)$ be a closed or a complete open Riemannian manifold
endowed with complementary orthogonal harmonic foliations
and a vector field $X$ such that $\|X\|_g\in L^1(M,g)$.
Suppose that $\,{\rm S}^{N,{\cal N}}_{\,{\rm mix},X}\ge0$ for some $N,{\cal N}<0$.
Then $M$~splits and $X=0$.
\end{thm}

\begin{proof}
Under conditions, from \eqref{E-PW-IF} we obtain
\begin{equation*}
 \frac12\Div X = {\rm S}^{N,{\cal N}}_{\,{\rm mix},X} +\|h^\bot\|^2 +\|h^\top\|^2
 -\frac1{2N}\,\|X^\top\|^2 -\frac1{2{\cal N}}\,\|X^\bot\|^2.
\end{equation*}
By Lemma~\ref{L-Div-1},
we get $\Div X=0$ when ${\rm S}^{N,{\cal N}}_{\,{\rm mix},X}\ge0$ and $N,{\cal N}<0$.
Thus, $h^\top=0=h^\bot$ and $X=0$. By de Rham decomposition theorem, $(M,g)$ splits.
\end{proof}

\subsubsection{Totally umbilical distributions}

If ${\cal D}^\bot$ is totally umbilical then $\|h^\bot\|^2-\|\Hb\|^2=-\frac{n-1}{n}\,\|\Hb\|^2$,
and similarly, for ${\cal D}^\top$.

\begin{thm}\label{C-Step3} Let $(M,g)$ be a closed (or a complete open) Riemannian manifold
endowed with complementary orthogonal totally umbilical distributions
${\cal D}^\top$ and ${\cal D}^\bot$ and a vector field $X$ obeying
$\|\xi_{|M}\|_{g}{\in} L^1(M,g)$, where $\xi={H}^\bot + {H}^\top + \frac12\,X$.
Suppose that $\,{\rm S}^{N,{\cal N}}_{\,{\rm mix},X}\le0$ for some $N,{\cal N}>0$.
Then $M$~splits and $X=0$.
\end{thm}

\begin{proof}
 Under conditions, from \eqref{E-PW-IF} we~get
\begin{equation}\label{E-umb-C5}
 \Div \xi = {\rm S}^{N,{\cal N}}_{\,{\rm mix},X}-\|T^\top\|^2 -\|T^\bot\|^2 -\frac{n{-}1}{n}\,\|\Hb\|^2 -\frac{\nu{-}1}{\nu}\,\|\Ht\|^2
 -\frac1{2N}\,\|X^\top\|^2 -\frac1{2{\cal N}}\,\|X^\bot\|^2.
\end{equation}
From \eqref{E-umb-C5} and Lemma~\ref{L-Div-1} and since ${\rm S}^{N,{\cal
 N}}_{\,{\rm mix},X}\le0$ for $N,{\cal N}>0$, we get $\Div\xi=0$.
The above yields vanishing of $T^\top,T^\bot,\Ht,\Hb$ and $X$. By de Rham decomposition theorem, $(M,g)$ splits.
\end{proof}

Umbilical integrable distributions appear on double-twisted products, see \cite{pr}.

\begin{definition}\rm
A \textit{doubly-twisted product} $B\times_{(v,u)} F$ of Riemannian manifolds $(B,g_B)$
and $(F, g_F)$ is a manifold $M=B\times F$ with metric $g = g^\top + g^\perp$, where
\begin{equation*}
 g^\top(X,Y) = v^2 g_B(X^\top,Y^\top),\quad g^\bot(X,Y)=u^2 g_F(X^\bot,Y^\bot),
\end{equation*}
and the warping functions $u,v\in C^\infty(M)$ are positive.
\end{definition}

Let ${\cal D}^\top$ be tangent to the \textit{fibers} $\{x\}\times F$ and ${\cal D}^\bot$
tangent to the \textit{leaves} $B\times\{y\}$. The second fundamental forms
and the mean curvature vectors of $B\times_{(v,u)} F$ are given by, see \cite{pr},
\begin{equation*}
  h^\bot = -\nabla^\top(\log u)\,g^\bot,\quad h^\top=-\nabla^\bot(\log v)\,g^\top,\quad
  \Hb=-n\,\nabla^\top(\log u),\quad \Ht=-\nu\,\nabla^\bot(\log v).
\end{equation*}
Thus, the {leaves} and the {fibers} of $B\times_{(v,u)} F$ are totally umbilical
with respect to $\bar\nabla$ and $\nabla$.

\begin{corollary}[of Theorem~\ref{C-Step3}]
Let $M=B\times_{(v,u)} F$ be complete closed (or complete open) and
there is a vector field $X$ such that
$\|\xi_{|M}\|_{g}{\in} L^1(M,g)$, where $\xi={H}^\bot + {H}^\top +\frac12\,X$.
Suppose that $\,{\rm S}^{N,{\cal N}}_{\,{\rm mix},X}\le0$ for some $N,{\cal N}>0$.
Then $M$~is the product and $X=0$.
\end{corollary}

\section{The weighted mixed sectional curvature}
\label{sec:w-Top}

Much is known about foliations with $K_{\rm mix}=\const$. Examples are totally geodesic
1)~$k$-{\it nullity} foliations on manifolds with degenerate curvature tensor,
the certain metrics are called {partially hyperbolic, parabolic or elliptic};
2)~{\it relative nullity} foliations of curvature-invariant submanifolds $M$ (e.g. of space forms).
Submanifolds with positive the relative nullity index $\mu(M)=\dim\ker h(m)$
(introduced by Chern and Kuiper)
have a structure of ruled developable submanifolds.

The proof of Ferus's result is based on analysis of the matrix Riccati equation
\begin{equation}\label{Eq-Ric}
 \dot B_{\dot\gamma}+(B_{\dot\gamma})^2 +R^\bot_{\dot\gamma}=0,
\end{equation}
along a leaf geodesic $\gamma$, where
$B_x=B(x,\cdot)$ for short, and
 the \textit{co-nullity
 tensor} of $\calf$ is
\[
 B(x,y)=-(\nabla_x\, y)^\bot\quad (x\in T\calf,\,y\in T^\bot\calf).
\]
 In this section, we introduce weighted modification of the co-nullity tensor of a totally geodesic foliated
manifold $(M,g)$ equipped with a vector field $X$,
\begin{equation*}
 B^X_{x} = B_x - g(X/n,x)\id^\bot.
\end{equation*}
Then the following ``weighted" Riccati equation holds along leaf geodesics:
\begin{equation}\label{E-Ric-Bf}
 \dot B^X_{\dot\gamma}+(B^X_{\dot\gamma})^2 +{2\,g(X/n,\dot\gamma)}\,B^X_{\dot\gamma} +R^{\bot}_{X,\dot\gamma}=0.
\end{equation}

Next theorem and corollary with constant $R^{\bot}_{X,x}>0$ generalize Ferus's results (with~$X=0$).

\begin{thm}\label{T-k-const}
Let $(M^{n+\nu}, g, X)$ be endowed with a totally geodesic foliation $\calf^\nu$.
Suppose that there exist $k=\const>0$ and a point $m\in M$ such that
along any leaf geodesic $\gamma:[0,\pi/\sqrt k]\to M$ with $\gamma(0)=m$
the weighted Jacobi operator has the view $R^{\bot}_{X,\dot\gamma}=k\id^\bot$ and
\begin{equation}\label{E-X1}
 g(X/n,\dot\gamma)^2 \le k.
\end{equation}
Then $\nu<\rho(n)$.
\end{thm}

\begin{proof} Assume the contrary, then there are unit vectors $x\in T_{m}\calf$ and $y\in T^\bot_{m}\calf$ and $\lambda_0\le0$ such that
$B^X_x\,y=\lambda_0 y$ for a geodesic $\gamma(t)$ with initial velocity $\dot\gamma(0)=x$.
Let $\bar y(t)\ (\bar y(0)=y)$ be a parallel vector field along $\gamma$.
The eigenvectors of the solution $B^X_{\dot\gamma}$ of \eqref{E-Ric-Bf} with $R^{\bot}_{X,\dot\gamma}=k\id^\bot$
do not depend on $t$.
Then $B^X_{\dot\gamma}\,\bar y(t)=\lambda(t)\,\bar y(t)$ for certain eigenfunction $\lambda(t)$, which satisfies the scalar Riccati equation
\begin{equation}\label{E-X2}
 \dot\lambda +\lambda^2 +
 2\,{\lambda}
 \,g(X/n,\dot\gamma) +k = 0.
\end{equation}
By \eqref{E-X1}, solution $\lambda(t)$ of \eqref{E-X2}  cannot be extended
 to the segment $[0,\pi/\sqrt k]$,
a contradiction.
\end{proof}

\smallskip

A submanifold $M\subset\bar M$ is called {\it curvature-invariant} if the curvature tensor of $\bar M$ obeys
\begin{equation}\label{E-6.5}
 \bar R_{x,y}\,z^{\bot}=0,\qquad(x,\,y,\,z\in TM).
\end{equation}
Such submanifold
with posi\-tive index of relative nullity $\mu(M)=\min\limits_{m\in M}\dim\ker h(m)$
has the structure of a ruled developable submanifold.
The {\it relative nullity space} of the second fundamental form $h$ of
$M\subset\bar M$ at $m\in M$ is given by
 $\ker h(m)=\{x\in T_mM:\,h(x,y)=0\ {\rm for\ all }\ y\in T_mM\}$.

 The {\it extrinsic $q$th Ricci curvature} is defined by formula
\[
 \Ric_h^q(x_0;x_1,\dots,x_q)=\sum\nolimits_{i=1}^q[g(h(x_0,x_0),h(x_i,x_i))-g(h(x_0,x_i),h(x_0,x_i))].
\]
For $q=1$ it is called an {\it extrinsic sectional curvature}.

\begin{corollary}[for $X=0$ see \cite{rov-m}] Let $M^n$ be a complete curvature-invariant submanifold in
$(\bar M^{n+p}, \bar g, X)$.
Suppose that along any geodesic $\gamma:\RR\to M$ starting at
$\dot\gamma(0)\in \ker h$,
the weighted Jacobi operator of $\bar M$ obeys $\bar R^{\bot}_{X,\dot\gamma}=k\id^\bot$,
and that \eqref{E-X1} holds for some constant $k>0$.
Then $M$ is a totally geodesic submanifold if any of the following requirements are satisfied:

\smallskip

1) $\mu(M)>\nu(n):=\max\{t:t<\rho(n-t)\}$,
\quad
2) ${\rm Ric}^q_h\le0$ and $2p<n-\nu(n)-q+\delta_{1q}$.

\end{corollary}

 In \cite{rov-m}, we examined Toponogov's conjecture for a foliation given near a complete leaf:
the necessity of additional assumptions in this case was shown, the conjecture was confirmed
for ruled submanifolds of spherical space forms,
rigidity and splitting of foliations with nonnegative mixed curvature was proven under suitable assumptions.
 The geometric construction, used in investigating of
 the problem, is based on estimates of the length and the volume of associated Jacobi field along an
``extremal geodesic" of given length, and it examines conditions when co-nullity tensor of
a foliation has no real eigenvectors; hence, the dimension of a leaf is smaller than~$\rho(n)$.
We~discovered a variation procedure based on the concept
of
the ``turbulence" of a foliation, that is the rotational component of the co-nullity tensor.
In this section  we extend the above methods for exploring the ``weighted" Toponogov's conjecture.

A smooth $(1,1)$-tensor field $Y(t):T^\bot_\gamma\calf\to T^\bot_\gamma\calf$ along a leaf geodesic $\gamma$ is called a \textit{Jacobi tensor} if it satisfies the Jacobi equation
\begin{equation}\label{Eq-3.3}
\ddot Y+R_{\dot\gamma}Y=0,
\end{equation}
and
$\ker Y(t)\cap\ker Y'(t)=\{0\}$ for all $t$;
hence,
the action of $Y$ on
linearly independent parallel sections of $\gamma^\bot$ gives rise to linearly independent Jacobi vector fields.
We have $ B_{\dot\gamma}=\dot Y Y^{-1}$.

A solution $y(t)\subset T^\bot_{\gamma(t)}\calf\approx\RR^n$ of the Jacobi equation $\ddot y+R_{\dot\gamma}\,y=0$
with constant
matrix $R_{\dot\gamma}=k\id^\bot>0$ has the view
$\ y(t)=y(0)\cos(\sqrt kt)+\frac{y'(0)}{\sqrt k}\sin(\sqrt kt)$.
If~the initial values $y(0)$ and $y'(0)$ are linearly independent,
then $y(t)$ parameterizes an ellipse in the plane $y(0)\land y'(0)$,
moreover, the area of the parallelogram $y(t)\land y'(t)$ is constant.

\begin{lem}[see \cite{rov-m}]\label{L-4.7}
Let a solution $y(t)\subset\Bbb R^n$ of the Jacobi ODE
\begin{equation}\label{E-4.1}
 \ddot y+R(t)y=0\quad(0\le t\le\pi/{\sqrt k}),
\end{equation}
be written in the form $y(t)=\bar y(t)+u(t)$,
where
$\bar y(t)=y(0)\cos(\sqrt k\,t)+\frac{y'(0)}{\sqrt k}\sin(\sqrt k\,t)$
and the norm $\ ||R(t)-k\id||\le\varepsilon_1< k/2$. Then
 $|u(t)|\le\frac{\varepsilon_1}{k-(1-\cos(\sqrt k\,t))\,\varepsilon_1}\int_0^t\sqrt k\,|\bar y(s)|\sin(\sqrt k\,(t-s))\,ds$\,.
\end{lem}

\begin{definition}[see \cite{rov-m}]\label{D-turb}\rm
The \textit{turbulence of a leaf} $L$ of a totally geodesic foliation
is defined~by
\[
 a(L)=\sup\{g(B_x(y),z):\,x\in TL,\ y,\,z\in TL^\bot,\ y\,\bot\,z,
\ |x|=|y|=|z|=1\}.
\]
\end{definition}

For Riemannian submersions, the turbulence was considered by B.\,O'Neill.
Notice that the condition $a(L)=0$ for all leaves means that
$T^\bot\calf$ is tangent to a totally umbilical foliation.

\smallskip
Next theorem with positive $R^{\bot}_{X,x}$  generalizes our result in \cite{rov-m} (when $X=0$).

\begin{thm}[Local]\label{T-4.18}
Let $\calf^\nu$ be a totally geodesic foliation of
$(M^{n+\nu},g,X)$,
and there exists a point $m\in M$ such that along any leaf geodesic
$\gamma:[0,\pi/{\sqrt k}]\to L\ (\gamma(0)=m)$ we have \eqref{E-X1} and
\begin{eqnarray}\label{E-4.34}
 && 0<k_1\id^\bot\le R^{\bot}_{X,\dot\gamma}\le k_2\id^\bot,\\
 \label{E-4.20}
 && (k_2-k_1+2\eps)\max\{a(L)^2,k\}\le 0.3k (k_2+\eps),
 \end{eqnarray}
where $k=\frac12\,(k_1+k_2)$ and
 $\eps:=
 \big\|\,g(\nabla_{\dot\gamma}\,(X/n),\dot\gamma) + g(X/n,\dot\gamma)^2\,\big\|<k_1$.
Then $\nu<\rho(n)$.
\end{thm}

\begin{proof} This is divided into steps. Notice that $\eps=0$ is provided by $X^\top=0$.

\smallskip
\textbf{Step 1}.
It is sufficient to show that linear
operators $B^X_x:T^\bot_m\calf\to T^\bot_m\calf,\ (x\ne 0)$,
do not have real eigenvalues.
Suppose the opposite, i.e., there exist unit vectors
$x_0\in T_m\calf,\ y_0\in T_m^\bot\calf$ and a real $\lambda\le 0$ with the property $B^X_{x_0}(y_0)=\lambda y_0$.
Let $\gamma(t):[0,\pi/{\sqrt k}]\to M,\ (\dot\gamma(0)=x_0)$ be a
leaf geodesic, and $y(t):\gamma\to T^\bot_\gamma\calf$ a Jacobi vector field
along $\gamma$ containing the vector $y_0$. Hence \eqref{E-4.1} holds
with $\|R(t)-k\id\|\le\frac{k_2-k_1}2+\eps$,
where we denote
$\ y'=\nabla_{\dot\gamma}\,y$ and $y''=\nabla_{\dot\gamma}\nabla_{\dot\gamma}\,y$.

The Jacobi vector field $y(t)$ may be written in a form
 $y(t)=\big(\cos(\sqrt k\,t)+\frac\lambda{\sqrt k}\sin(\sqrt k\,t)\big)y_0+u(t)$,
 where
 $u(0)=u'(0)=0$.
We will show for \eqref{E-4.34} and $k_1-\eps\ge 0.582 (k_2+\eps)$,
that the function $|y(t)|$ (the length of the Jacobi vector field  $y(t)$)
has a local minimum at $t_m$ in the interval $(0,\pi/{\sqrt k})$.
 The second observation is that the function
$V(t)$ -- the area of a parallelogram, whose sides are the vectors $y(t)$ and $y'(t)$,
varies ``slowly" along a geodesic $\gamma$.
(This function is constant when $k_2=k_1$.)
 These will yield a contradiction, because $V(t)$ cannot increase from zero value $V(0)$ to a ``large" value $V(t_m)$
 on a given interval with length $t_m<\pi/{\sqrt k}$.

\smallskip

\textbf{Step 2}.
We shall prove that the
inequality
\begin{equation}\label{E-4.22}
 |V(t)'|\le\big(\frac12\,(k_2-k_1)+\eps\big)|y(t)|^2.
\end{equation}
The derivative of function $V^2=y^2(y')^2-(y,y')^2$, in view of \eqref{E-4.1}, is the following:
\[
 (V^2)'=-2\ \vmatrix g(R(t)y,y') & g(R(t)y,y)\\
 g(y,y') & g(y,y)\endvmatrix .
\]
Using linear combinations of columns in this $2\times 2-$determinant, we obtain
\begin{equation}\label{E-4.23}
 (V^2)'=-2\,g(R(t)y,\widetilde{y'})\,y^2 = -2\,g(\widetilde{R(t)y},\widetilde{y'})\,y^2,
 \end{equation}
where $\ \widetilde{}\ $ denotes the component orthogonal to vector $y(t)$. Since
\begin{equation}\label{E-4.24}
 (V(t)^2)'=2V'(t)V(t),\quad  V(t)=|y(t)|\cdot|\widetilde{y'}(t)|,
\end{equation}
then from \eqref{E-4.23} follows $\ |V(t)'|\le|\widetilde{R(t)\,y}(t)|\cdot|y(t)|$.
Thus we obtain \eqref{E-4.22} with the help of
\[
 |\widetilde{R\,y}(t)|\le|(R(t)-k\id)y(t)|\le \big(\frac12\,({k_2-k_1})+\eps\big)|y(t)|.
\]

\textbf{Step 3}. Assume that the first local minimum $t_m$ of $|y(t)|$
for $t\ge 0$ belongs to $(0,\pi/{\sqrt k}]$. Set $\tilde m=\gamma(t_m)$.
Consider (in this step) the opposite parametrization of $\gamma$ (and objects on it)
with parameter $s=t_m-t$, i.e., $\gamma(0)=\tilde m,\ \gamma(t_m)=m$.
Since $|y(s)|^2$ has a local minimum at $\tilde m$, then
\[
 y'(0)\,\bot\,y(0),\quad g(y,y)''\ge 0,
\]
and from \eqref{E-4.20} and equality $\frac 12\,g(y,y)''=g(y,y'')+g(y',y')$ follows
\begin{equation}\label{E-4.25}
 V(0)=|y'(0)|\cdot|y(0)|\ge\sqrt{k_1-\eps}\,|y(0)|^2.
\end{equation}
Moreover, $\,g(y'(\tilde m),y'(\tilde m))=g(B^X_{\dot\gamma}(y(\tilde m)),y'(\tilde m))$, i.e.,
\[
 \frac{|y'(\tilde m)|}{|y(\tilde m)|}\le\big|\big(B^X_{\dot\gamma}\big(\frac{y(\tilde m)}{|y(\tilde m)|}\big),\frac{y'(\tilde m)}{|y'(\tilde m)|}\,\big)\big|
 \le a(L).
\]
We can write
$y(s)=Y_1(s)+u_1(s)$, where
$Y_1(s)=y(0)\cos(\sqrt k\,s)+\frac{y'(0)}{\sqrt k}\sin(\sqrt k\,s).$
By Lemma~\ref{L-4.7} with $\varepsilon_1=\frac12\,({k_2-k_1})+\eps$, we have
\begin{eqnarray*}
 && 1 = |y(\tilde m)|\le|Y_1(t_m)|+|u_1(t_m)| \\
 && \le \max\{{a(L)}/{\sqrt k},1\}|y(\tilde m)|\Big(1+(1-\cos(\sqrt k\,t_m))\,
\frac{k_2-k_1+2\eps}{3k_1-k_2-4\eps}\Big),
\end{eqnarray*}
i.e., the inequality
\begin{equation}\label{E-4.26}
 |y(\tilde m)|\ge\frac{3k_1-k_2-4\eps}{k_1+k_2}/\max\big\{{a(L)}/{\sqrt k},1\big\}.
\end{equation}
From \eqref{E-4.20} follows $k_1-\eps\ge 0.7 (k_2+\eps)$,
but at the same time for $3k_1-k_2-4\eps>0$ it is sufficient to require $k_1-\eps\ge\frac 13\,(k_2+\eps)$.
From eigenvector's condition $B^X_{x_0}(y_0)=\lambda y_0$ follows
$V(t_m)=0$. In view of \eqref{E-4.22} and the estimate $|y(s)|\le 1$ for $0\le s\le t_m$, we obtain
\[
 V(0)\le\int_0^{t_m}|V'(s)|\,ds\le \frac12\,(k_2-k_1+2\eps)\int_0^{t_m}|y(s)|^2\,ds\le\big(\frac12\,({k_2-k_1})+\eps\big)t_m.
\]
From the above and \eqref{E-4.25}, \eqref{E-4.26} follows the inequality
\begin{equation}\label{E-4.27}
 \Big(\frac\pi2+\tau\Big)(1-\delta)
 \Big(\frac{1+\delta}{3\delta-1}\Big)^2\max\{{a(L)^2}/k,1\}\ge\sqrt{2\delta(1+\delta)},
\end{equation}
where
 $\tau=\sqrt k\,(t_m-t_0),\
 \delta=\frac{k_1-\eps}{k_2+\eps}$
 and
 $\sqrt k\,t_0=\text{arcctg}\big(-\lambda/{\sqrt k}\big)\in(0,\pi/2]$.

\textbf{Step 4}. We now go back to initial parametrization of $\gamma$ and
consider the problem, when $|y(t)|$ has a local minimum in the interval $(0,\pi/{\sqrt k}]$.
We can write $y(t)=\bar y(t)+u(t)$, where $t\in[t_0,2t_0]$ and
\[
 \bar y(t)=(\cos(\sqrt k\,t)+\frac\lambda{\sqrt k}\,\sin(\sqrt k\,t))y_0
 =\frac{\sin(\sqrt k\,(t_0-t))}{\sin(\sqrt k\,t_0)}\,y_0.
\]
By Lemma~\ref{L-4.7} with $\varepsilon_1=\frac12\,({k_2-k_1})+\eps$, we obtain
\begin{equation}\label{E-4.28}
 |u(t)|\le\frac{k_2-k_1+2\eps}{3k_1-k_2-4\eps}\cdot\frac 1{\sin(\sqrt k\,t_0)}
\int_0^t\sqrt k\sin(\sqrt k\,(t-s))\sin(\sqrt k\,|t_0-s|)\,ds.
\end{equation}
Using the trigonometry identity
\[
 f=\sin(\sqrt k\,(t-s))\sin(\sqrt k(t_0-s))=\frac12\,({\cos(\sqrt k(t-t_0))-\cos(\sqrt k(t+t_0-2s))}),
\]
and the shortland notation
 $\tau=\sqrt k\,(t-t_0)$,
 $\tau_0=\sqrt k\,t_0$,
 $S=\sin$,
 and
 $C=\cos$,
we transform the integral in RHS of \eqref{E-4.28}:
\[
 I(t)=\sqrt k\,\Big(\int_0^{t_0}fds-\int_{t_0}^tfds\Big)= \frac12\,(\tau_0-\tau){C(\tau)} +\frac34\,{S(\tau )}-\frac14\,{S(\tau+2\tau_0)}.
\]
Since
\[
 \tau_0,\ \tau\in[0,\pi/2],\ \ \tau_0-\tau\ge 0,\ \ C(\tau)\ge0,\ \ S(2\tau_0+\tau)\ge S(\pi+\tau)=-S(\tau),
\]
then the function $I(t)$ has the largest upper value for $\tau_0=\pi/2$:
\[
 I(t)\le\Big(\frac\pi2-\tau\Big)\frac{C(\tau)}2+S(\tau),\quad I(t_0)\le\frac\pi4\text{\ \ (since\ \ }\tau=0).
\]
Consequently, assuming $\delta=\frac{k_1}{k_2}$,
\[
 |u(t)|\le\Big((\frac\pi2-\tau)\frac{C(\tau)}2+S(\tau)\Big)
 \frac{1-\delta}{(3\delta-1)S(\tau_0)},\quad |u(t_0)|\le\frac\pi4\ \frac{1-\delta}{(3\delta-1)S(\tau_0)}.
\]
Note that $|\bar y(t)|=\frac{S(\tau)}{S(\tau_0)}$ and $\bar y(t_0)=0$.
Since $|y|'(0)=\lambda\le 0$, and in case of $\lambda=0$ from \eqref{E-4.34}
follows that $|y|''(0)<0$. Hence, $|y(t)|$ is decreasing for
small values $t\ge 0$. Thus for the property that $|y(t)|$ has a
local minimum at some $t_m\in(0,t]$ it is sufficient to require
$|y(t_0)|\le|y(t)|$, i.e.,
\[
 |u(t_0)|+|u(t)|\le|\bar y(t)|.
\]
The last inequality (in view of above estimates) is reduced to
\[
 \big(\frac\pi4+\frac12\big(\frac\pi2-\tau\big){C(\tau)}+S(\tau)\big)\frac{1-\delta}{3\delta-1}\le S(\tau),
\]
that is equivalent to
\begin{equation*}
 \delta\ge 1-\frac{2S(\tau)}{\frac\pi4+\frac12(\frac\pi2-\tau){C(\tau)}+4S(\tau)}.
\end{equation*}
Notice that \eqref{E-4.20} yields $\delta\ge 0.7$.
One can verify that inequality
 $\frac{2S(\tau)}{\frac\pi4+(\frac\pi2-\tau)\frac{C(\tau)}2+4S(\tau)}\ge 0.3$
holds for all $\tau\in[0.5,\pi/2]$, i.e., for $\delta\ge 0.7$ the function
$|y(t)|$ has local minimum at $t_m\in[0,\ (0.5+\pi/2)/\sqrt k]$.
But for $\delta\in[0.7,1]$ and $\tau=0.5$ the inequality
\begin{equation}\label{E-4.30}
 0.3\Big(\frac\pi2+\tau\Big)\Big(\frac{1+\delta}{3\delta-1}\Big)^2<\sqrt{2\delta(1+\delta)}
 \end{equation}
is equivalent to
\begin{equation}\label{E-4.31}
 f(\delta)=\Big(\frac{3\delta-1}{1+\delta}\Big)^2\sqrt{2\delta(1+\delta)}>0.15(\pi+1).
 \end{equation}
Since $\frac{3\delta-1}{1+\delta}$ is monotone increasing, then
also $f(\delta)$ is monotone increasing for $\delta>1/3$. It is
easy to show $f(0.7)>0.63>0.15(\pi+1)$. Thus \eqref{E-4.31} and also \eqref{E-4.30}
are true for $\delta\ge 0.7$. From \eqref{E-4.30} and \eqref{E-4.27} follows the
inequality, which contradicts to \eqref{E-4.20}.
\end{proof}

Next theorem with nonnegative $R^{\bot}_{X,x}$
generalizes our result in \cite{rov-m} (when $X=0$).

\begin{thm} [Decomposition]\label{T-decomp}
Let $\calf^\nu$ be a compact totally geodesic foliation of $(M^{n+\nu},g,X)$.
Suppose that \eqref{E-X1} and
\begin{eqnarray*}
 && 0\le k_1\id^\bot\le R^{\bot}_{X,x}\le k_2\id^\bot\quad(x\in T\calf,\ |x|=1),\\
 && (k_2-k_1+2\,\eps)\cdot\max\{a(L)^2,k\}\le0.337\,(k_2+\eps)\,k,
\end{eqnarray*}
 where $k=\frac12\,(k_1+k_2)$,
$\eps:=\big\|\,g(\nabla_{x}\,(X/n), x)+g(X/n, x)^2\,\big\|<k_1$ for all unit $x\in T\calf$, and $L$ is some leaf.
 If~$\nu\ge\rho(n)$ then $k_1=k_2=0$ and $M$ splits along $\calf$.
\end{thm}



\end{document}